\documentclass[11pt]{amsart}
\usepackage{graphicx,amssymb,amsmath,amsthm}
\usepackage{enumerate}
\usepackage{dsfont}
 \numberwithin{equation}{section}
\usepackage{mathrsfs}
\usepackage{epsfig}
\usepackage{lscape}
\usepackage{subfigure}
\usepackage{graphicx}
\usepackage{epstopdf}
\usepackage{color}
\usepackage{caption}
\usepackage{algorithm}
\usepackage{algpseudocode}
\usepackage{bm}

\allowdisplaybreaks[4]

\newcommand{\abs}[1]{\lvert#1\rvert}

\newcommand{\argmin}[1]{\mathop{\rm argmin}\limits_{#1}}

\newcommand{\A}{{\mathcal A}}
\newcommand{\E}{{\mathbb E}}

\newcommand{\R}{{\mathbb R}}

\newcommand{\HH}{{\mathbb F}}

\newcommand{\C}{{\mathbb C}}
\newcommand{\CU}{{\mathcal U}}

\newcommand{\vx}{{\mathbf x}}

\newcommand{\vy}{{\mathbf y}}
\newcommand{\vu}{{\mathbf u}}

\newcommand{\va}{{\mathbf a}}

\newcommand{\innerp}[1]{\langle #1 \rangle}

\newcommand{\F}{{\mathbb F}}
\newcommand{\Z}{{\mathbb Z}}

\newcommand{\diag}{{\rm diag}}

\renewcommand{\omega}{\eta}

\newcommand{\RNum}[1]{\uppercase\expandafter{\romannumeral #1\relax}}

\newtheorem{definition}{Definition}[section]
\newtheorem{corollary}[definition]{Corollary}

\newtheorem{theorem}[definition]{Theorem}
\newtheorem{lemma}[definition]{Lemma}

\date{}

\begin{document}
\baselineskip 18pt
\bibliographystyle{plain}
\title{The recovery of complex sparse signals from few phaseless measurements }
\author{Yu Xia}
\thanks{Yu Xia was supported by NSFC grant (11901143), Zhejiang Provincial Natural Science Foundation (LQ19A010008), Education Department of Zhejiang Province Science Foundation (Y201840082). }
\address{Department of Mathematics,
Hangzhou Normal University, Hangzhou 311121, China
}
\email{yxia@hznu.edu.cn}

\author{Zhiqiang Xu}
\thanks{Zhiqiang Xu was supported  by NSFC grant (91630203, 11688101),
Beijing Natural Science Foundation (Z180002).
}
\address{LSEC, Inst.~Comp.~Math., Academy of
Mathematics and System Science,  Chinese Academy of Sciences, Beijing, 100091, China
\newline
School of Mathematical Sciences, University of Chinese Academy of Sciences, Beijing 100049, China}
\email{xuzq@lsec.cc.ac.cn}
\begin{abstract}
We study the stable  recovery of complex $k$-sparse signals from as few phaseless
measurements as possible. The main result is to show that  one can employ $\ell_1$
minimization to  stably recover complex $k$-sparse signals from $m\geq O(k\log
(n/k))$   complex Gaussian random quadratic measurements with high probability. To do
that, we establish that Gaussian random measurements  satisfy the
 restricted isometry property over rank-$2$ and sparse matrices with high probability.
 This paper presents the first  theoretical estimation of the measurement number
for stably recovering complex sparse signals  from   complex Gaussian quadratic
measurements.
\end{abstract}
\maketitle
\section{Introduction}
\subsection{Compressive Phase Retrieval }
Suppose that $\vx_0\in \HH^n$ is a  $k$-sparse signal, i.e., $\|\vx_0\|_0\leq k$,
where  $\HH\in \{\R,\C\}$. We are interested in recovering $\vx_0$ from
\[
y_j=\abs{\innerp{\va_j,\vx_0}}^2+w_j,\quad j=1,\ldots,m,
\]
where $\va_j\in \HH^n$ is a measurement vector and $w_j\in \R$ is the noise.
This problem is called {\em compressive phase retrieval } \cite{CPR07}.
  To state conveniently, we let $\A: \F^{n\times n}\rightarrow \R^m$ be a linear map which is defined as
\[
\A(X)=(\va_1^*X\va_1,\ldots,\va_m^*X\va_m),
\]
where $X\in \F^{n\times n}, \va_j\in \F^n, j=1,\ldots,m$. We abuse the notation and
set
\[
\A(\vx):=\A(\vx\vx^*)=(\abs{\innerp{\va_1,\vx}}^2,\ldots,\abs{\innerp{\va_m,\vx}}^2),
\]
where $\vx\in \HH^n$.
We also set
\[
\tilde{\vx}_0\,\,:=\,\, \{c \vx_0: \abs{c}=1, c\in \HH\}.
\]
The aim of compressive phase retrieval is to recover $\tilde{\vx}_0$ from noisy
measurements $\mathbf{y}=\A(\vx_0)+\mathbf{w}$, with
$\mathbf{y}=(y_1,\ldots,y_m)^T\in \R^m$ and $\mathbf{w}=(w_1,...,w_m)^T\in \R^m$. One
question in compressive phase retrieval is: {\em how many measurements $y_j,
j=1,\ldots,m$, are needed to stably recover $\tilde{\vx}_0$?} For the case $\HH=\R$,
 in \cite{EM}, Eldar and Mendelson established that $m=O(k\log (n/k))$ Gaussian random quadratic
 measurements are enough to stably recover $k$-sparse signals $\tilde{\vx}_0$.
 For the complex case, Iwen,  Viswanathan and  Wang suggested a two-stage strategy for
  compressive phase retrieval and show that $m=O(k\log (n/k))$ measurements can guarantee
   the stable recovery of $\tilde{\vx}_0$ \cite{Iwen2015Robust}. However, the strategy
   in \cite{Iwen2015Robust} requires the measurement matrix  to be written as a
   product    of two random matrices. Hence, it still remains open whether one can  stably recover
    arbitrary complex $k$-sparse signal $\tilde{\vx}_0$ from $m=O(k\log (n/k))$  Gaussian random
     quadratic measurements. One of the aims of this paper is to confirm  that
     $m=O(k\log (n/k))$ Gaussian random quadratic measurements are enough to guarantee the
      stable recovery of arbitrary  complex $k$-sparse signals. In fact,
we do so by  employing $\ell_1$ minimization.

\subsection{$\ell_1$ minimization }
Set $A:=(\va_1,\ldots,\va_m)^T\in \HH^{m\times n}$.  One classical result in
compressed sensing is that one can use $\ell_1$ minimization to recover $k$-sparse
signals, i.e.,
\[
\argmin{\vx\in \HH^n} \{\|\vx\|_1 : A\vx=A\vx_0\}=\vx_0,
\]
provided that the measurement matrix $A$ meets the RIP condition \cite{CJT}. Recall
that a matrix $A$ satisfies the $k$-order RIP condition with RIP constant
$\delta_k\in [0,1)$ if
\[
(1-\delta_k)\|\mathbf{x}\|_2^2\leq \|A\mathbf{x}\|_2^2\leq (1+\delta_k)\|\mathbf{x}\|_2^2
\]
holds for all $k$-sparse vectors $\mathbf{x}\in \mathbb{F}^n$. Using tools from
probability theory, one can show that Gaussian random matrix satisfies $k$-order RIP
with high probability provided $m=O(k\log (n/k))$ \cite{BDDW08}.

Naturally, one is interested in employing $\ell_1$ minimization for compressive phase retrieval.
We consider the following model:
\begin{equation}\label{eq:l1model}
\argmin{\vx\in \HH^n}\{\|\vx\|_1: \abs{A\vx}=\abs{A\vx_0}\}.
\end{equation}
Although the constrained conditions in (\ref{eq:l1model}) is non-convex, the model
(\ref{eq:l1model}) is more amenable to algorithmic recovery. In fact, one already
develops many algorithms for solving (\ref{eq:l1model}) \cite{alg1,alg2}. For the
case $\HH=\R$, the performance of (\ref{eq:l1model}) was studied in
\cite{VX16,GWX,WX14,Lu19}. Particularly, in \cite{VX16}, it was shown that if $A\in
\R^{m\times n}$ is a random Gaussian matrix with  $m=O(k\log (n/k))$, then
\[
\argmin{\vx\in \R^n}\{\|\vx\|_1: \abs{A\vx}=\abs{A\vx_0}\}\,\,=\,\,\pm\vx_0
\]
holds with high probability. The methods developed in \cite{VX16} heavily depend on
$A\vx_0$ is a {\em real} vector and one still does not know the  performance of
$\ell_1$ minimization for recovering complex sparse signals. As mentioned in
\cite{VX16}: ``{\em The extension of these results to hold over $\C$ cannot follow
the same line of reasoning}". In this paper, we extend the result in \cite{VX16} to
complex case with employing the new idea about the RIP of quadratic measurements.

\subsection{Our contribution }
In this paper, we study the performance of $\ell_1$ minimization for recovering complex sparse signals from phaseless measurements $\mathbf{y}=\mathcal{A}(\mathbf{x}_0)+\mathbf{w}$, where $\|\mathbf{w}\|_2\leq \epsilon$. Particularly, we focus on the model
\begin{equation}\label{eq:complexL1}
\min_{\vx\in \C^n} \|\vx\|_1 \quad {\rm s.t.}\quad \|\A(\vx)-\A(\vx_0)\|_2\leq \epsilon.
\end{equation}
Our main idea  is to lift  (\ref{eq:complexL1}) to recover  rank-one and sparse matrices, i.e.,
 \[
\min_{X\in {\mathbb H}^{n\times n}} \|X\|_1\quad {\rm s.t.}\quad \|\A(X)-\vy\|_2\leq \epsilon, \  {\rm rank}(X)=1.
\]
Throughout this paper, we use ${\mathbb H}^{n\times n}$ to denote the $n\times n$
Hermitian matrices set.
 Thus we require $\A$ satisfies restricted isometry property over low-rank and sparse matrices:
\begin{definition}
We say that the  map $\A:{\mathbb H}^{n\times n}\rightarrow \R^m$ satisfies the
restricted isometry property of order $(r,k)$ if there exist positive constants $c$
and $C$ such that the inequality
\begin{equation}\label{de:RIP}
c\|X\|_{F}\leq\frac{1}{m}\|\mathcal{A}(X)\|_{1}\leq C\|X\|_{F}
\end{equation}
holds for all $X\in \mathbb{H}^{n\times n}$ with $\text{rank}(X)\leq r$ and
$\|X\|_{0,2}\leq k$.
\end{definition}
Throughout this paper, we use $\|X\|_{0,2}$ to denote the number of non-zero rows in
$X$. Since $X$ is Hermitian, we have $\|X\|_{0,2}=\|X^*\|_{0,2}$. We next show that
Gaussian random map $\A$ satisfies RIP of order $(2,k)$ with high probability
provided $m\gtrsim k\log (n/k)$.
\begin{theorem}\label{thm: RIP}
Assume that the linear measurement $\mathcal{A}(\cdot)$ is defined as
\[
\mathcal{A}(X)=(\va_1^*X\va_1,\ldots,\va_m^*X\va_m),
\]
with  $\va_j$ independently taken as complex Gaussian random vectors, i.e.,
$\va_j\sim\mathcal{N}(0, \frac{1}{2}\mathbf{I}_{n\times
n})+\mathcal{N}(0,\frac{1}{2}\mathbf{I}_{n\times n})i$. If
\[
 m\gtrsim k\log (n/k),
 \]
 under the probability at least $1-2\exp(-c_0m)$, the linear map $\mathcal{A}$
satisfies the restricted isometry property  of order $(2,k)$, i.e.
\[
0.12\|X\|_{F}\leq\frac{1}{m}\|\mathcal{A}(X)\|_{1}\leq 2.45\|X\|_{F},
\]
for all $X\in \mathbb{H}^{n\times n}$ with $\text{rank}(X)\leq 2$ and
$\|X\|_{0,2}\leq k$ (also $\|X^{*}\|_{0,2}\leq k$).
\end{theorem}

In the next theorem, we present the performance of (\ref{eq:complexL1}) with showing
that one can employ $\ell_1$ minimization to   stably recovery complex $k$-sparse
signals from the  phaseless measurements provided $\A$ satisfies restricted isometry
property of  order $(2,2ak)$ with a proper choice of $a>0$.

\begin{theorem}\label{thm: noise_constrained}
Assume that $\mathcal{A}(\cdot)$ satisfy the RIP condition of order $(2,2ak)$ with
RIP constant $c, C>0$ satisfying
\begin{equation}\label{eqn: condition_a}
c-\frac{4C}{\sqrt{a}}-\frac{C}{a}>0.
\end{equation}
For any $k$ sparse signals $\vx_0\in \C^n$,  the solution to (\ref{eq:complexL1})
$\vx^\#$ satisfies
\begin{equation}\label{eqn: conclusion1}
\|\vx^\#(\vx^\#)^* -\vx_0\vx_0^*\|_2\leq C_1\frac{2\epsilon}{\sqrt{m}},
\end{equation}
where
\[
C_{1}=\frac{\frac{1}{a}+\frac{4}{\sqrt{a}}+1}{c-\frac{4C}{\sqrt{a}}-\frac{C}{a}}.
\]
Furthermore, we have
\begin{equation}\label{eqn: conclusion2}
\min_{c\in \C, \abs{c}=1}\|c\cdot\vx^\# -\vx_0\|_2\leq \min\{ 2\sqrt{2}C_{1}\epsilon/ (\sqrt{m}\|\vx_0\|_2),
2\sqrt{2\sqrt{2}C_1}\cdot\sqrt{\epsilon}\cdot (n/m)^{1/4}\}.
\end{equation}
\end{theorem}
 According to  Theorem \ref{thm: RIP},  if $\va_j, j=1,\ldots,m$ are complex Gaussian random vectors,
 then  $\A$ satisfies RIP of order $(2,2ak)$
with constants $c=0.12$ and $C=2.45$ with high probability provided $m\gtrsim 2ak
\log(n/2ak)$. To guarantee (\ref{eqn: condition_a}) holds, it is enough to require
$a>(8C/c)^2$. Therefore, combining Theorem \ref{thm: RIP} and Theorem \ref{thm:
noise_constrained} with $\epsilon=0$, we can obtain the following corollary:
\begin{corollary}
Suppose that $\vx_0\in \C^n$ is a $k$-sparse signal.
Assume that $A=(\va_1,\ldots,\va_m)^T$ where $\va_j, j=1,\ldots,m$ is Gaussian random vectors, i.e.,  $\va_j \sim\mathcal{N}(0, \frac{1}{2}\mathbf{I}_{n\times n})+\mathcal{N}(0,\frac{1}{2}\mathbf{I}_{n\times n})i$. If
$ m\gtrsim k\log(n/k)$, then
\[
\argmin{\vx\in \C^n}\{\|\vx\|_1: \abs{A\vx}=\abs{A\vx_0}\}\,\,=\,\,\tilde{\vx}_0
\]
holds under the probability at least $1-2\exp(-c_0m)$. Here $c_0>0$ is an absolute
constant.

\end{corollary}

\section{Proof of Theorem \ref{thm: RIP}}

 We first introduce Bernstein-type inequality which
plays a key role in our proof.
\begin{lemma}\cite{V10}
\label{Bernstein inequality} Let $\xi_1,\ldots,\xi_m$ be i.i.d. sub-exponential
random variables and $K:=\max_j \|\xi_j\|_{\psi_1}$. Then for every $\epsilon>0$, we
have
\[
\mathbb{P}\left(\left|\frac{1}{m}\sum_{j=1}^m\xi_j-\frac{1}{m}\mathbb{E}(\sum_{j=1}^m\xi_j)\right|\geq \epsilon\right)\leq 2\exp\left(-c_0 m\min\left(\frac{\epsilon^2}{K^2},\frac{\epsilon}{K}\right)\right),
\]
where $c_0>0$ is an absolute constant.
\end{lemma}
We next  introduce some key lemmas needed to prove Theorem \ref{thm: RIP}, and then
present the proof of Theorem \ref{thm: RIP}.

\begin{lemma}
\label{lem: expectation} If $z_1$, $z_2$, $z_3$ and $z_4$ are independently drawn
from $\mathcal{N}(0,1)$. When $t\in [-1,0]$, we have
\[
\mathbb{E}|z_1^2+z_2^2+tz_3^2+tz_4^2|=2\left(\frac{1+t^2}{1-t}\right).
\]
\end{lemma}
\begin{proof}
When $t=0$, we have
$\mathbb{E}|z_1^2+z_2^2+tz_3^2+tz_4^2|=\mathbb{E}|z_1^2+z_2^2|=2$. We consider  the
case where $-1\leq t<0$.
 Taking coordinates transformation as $z_{1}=\rho_{1}\cos\theta$,
$z_{2}=\rho_{1}\sin\theta$, $z_{3}=\rho_{2}\cos\phi$, and $z_{4}=\rho_{2}\sin\phi$,
we obtain that
\begin{align*}
\mathbb{E}|z_{1}^{2}+z_{2}^{2}+tz_{3}^{2}+tz_{4}^{2}| &
=\left(\frac{1}{2\pi}\right)^{2}\int_{{\mathbb R}^4}|z_{1}^{2}+z_{2}^{2}+tz_{3}^{2}+tz_{4}^{2}|\exp\left(-\frac{z_{1}^{2}+z_{2}^{2}+z_{3}^{2}+z_{4}^{2}}{2}\right)dz_{1}dz_{2}dz_{3}dz_{4}\\
 & =\left(\frac{1}{2\pi}\right)^{2}\int_{0}^{2\pi}d\theta\int_{0}^{2\pi}d\phi\int_{0}^{\infty}\int_{0}^{\infty}\rho_{1}\rho_{2}|\rho_{1}^{2}+t\rho_{2}^{2}|\exp\left(-\frac{\rho_{1}^{2}+\rho_{2}^{2}}{2}\right)d\rho_{1}d\rho_{2}\\
 & =\int_{0}^{\infty}\int_{0}^{\infty}\rho_{1}\rho_{2}|\rho_{1}^{2}+t\rho_{2}^{2}|\exp\left(-\frac{\rho_{1}^{2}+\rho_{2}^{2}}{2}\right)d\rho_{1}d\rho_{2}\\
 & =\int_{\rho_{1}>\sqrt{-t}\rho_{2}}\rho_{1}\rho_{2}(\rho_{1}^{2}+t\rho_{2}^{2})\exp\left(-\frac{\rho_{1}^{2}+\rho_{2}^{2}}{2}\right)d\rho_{1}d\rho_{2}\\
 &\quad +\int_{\rho_{1}\leq\sqrt{-t}\rho_{2}}\rho_{1}\rho_{2}\left(-t\rho_{2}^{2}-\rho_{1}^{2}\right)\exp\left(-\frac{\rho_{1}^{2}+\rho_{2}^{2}}{2}\right)d\rho_{1}d\rho_{2}\\
 &= \frac{2}{1-t}+\frac{2t^{2}}{1-t}=\frac{2(1+t^{2})}{1-t}.
\end{align*}
Here, we evaluate the last integrals using the integration by parts.
 One can also use {\tt Maple } to check the integrals.

\end{proof}

\begin{lemma}
\label{lem: covering number} Set
\[
\mathcal{X}:=\{X\in \mathbb{H}^{n\times n}\ |\ \|X\|_F=1,\ \text{rank}(X)\leq 2,\ \|X\|_{0,2}\leq k\}
\]
which is equipped with Frobenius norm.   The covering number of ${\mathcal X}$ at
scale $\epsilon>0$ is less than or equal to {{$ \left(\frac{9\sqrt{2}en}{\epsilon
k}\right)^{4k+2}$}}.
\end{lemma}
\begin{proof}
 Note that
\[
\mathcal{X}=\{X\in \mathbb{H}^{n\times n}\ : \ X=U\Sigma U^*, \ \Sigma\in \Lambda,\ U\in \mathcal{U}\},
\]
where
\[
\Lambda=\{\Sigma\in \mathbb{R}^{2\times 2}\ : \ \Sigma=\diag(\lambda_1,\lambda_2),\ \lambda_1^2+\lambda_2^2=1\}
\]
and
\[
\mathcal{U}=\{U\in \mathbb{C}^{n\times 2}\ :\ U^*U=I,\ \|U\|_{0,2}\leq k\}
=\cup_{\# T =k}\CU_T.
\]
Here $T\subset \{1,\ldots,n\}$, and
\[
\CU_T:=\{U\in \mathbb{C}^{n\times 2}\ :\ U^*U=I,\ U=U_{T,:}\}.
\]
The $U_{T,:}\subset \C^{n\times 2}$ is a matrix obtained by keeping the rows of $U$
indexed by $T$ and setting the rows off the index set $T$ as $0$.
Note that $\|U\|_F=\sqrt{2}$ for all $U\in \mathcal{U}_T$ and  that the real
dimension of $\CU_T$ is at most $4k$ for any fixed support $T$ with $\# T= k$.
 We use $Q_T$ to denote $\epsilon/3$-net of $\CU_T$ with $\#Q_T\leq (9\sqrt{2}/\epsilon)^{4k}$.
 Then $Q_\epsilon:=\cup_{\# T=k}Q_{T}$ is a $\epsilon/3$-net of $\mathcal{U}$
 with
\[
\#Q_\epsilon\leq \left(\frac{en}{k}\right)^k\left(\frac{9\sqrt{2}}{\epsilon}\right)^{4k}\leq
\left(\frac{9\sqrt{2}en}{\epsilon k}\right)^{4k}.
\]
 We use $\Lambda_\epsilon$ to denote the $\epsilon/3$-net of $\Lambda$  with
$\#\Lambda_\epsilon\leq (9/\epsilon)^2$.

Set
\[
\mathcal{N}_\epsilon\,\,:=\,\,\{U\Sigma U^*\ |\ U\in Q_\epsilon,\ \text{and}\  \Sigma \in \Lambda_\epsilon\}.
\]
Then for any $X=U\Sigma U^*\in \mathcal{X}$, there exists $U_0\Sigma_0 U_0^*\in
\mathcal{N}_\epsilon$ with $\|U-U_0\|_F\leq \epsilon/3$ and
$\|\Sigma-\Sigma_0\|_F\leq \epsilon/3$. So, we have
\[
\begin{split}
\|U\Sigma U^*-U_0\Sigma_0 U_0^*\|_F&\leq \|U\Sigma U^*-U_0\Sigma U^*\|_F+\|U_0\Sigma U^*-U_0\Sigma_0 U^*\|_F+\|U_0\Sigma_0 U^*-U_0\Sigma_0 U_0^*\|_F\\
&\leq  \|U-U_0\|_F\|\Sigma U^*\|+\|U_0\|\|\Sigma-\Sigma_0\|_F\| U^*\|+\|U_0\Sigma_0\| \|U^*-U_0\|_F\\
&\leq \epsilon.
\end{split}
\]
Therefore, $\mathcal{N}_\epsilon$ is an $\epsilon$-net of $\mathcal{X}$ with
{{\[\#\mathcal{N}_\epsilon\leq \#\mathcal{Q}_\epsilon\cdot \#{\Lambda}_\epsilon\leq \left(\frac{9\sqrt{2}en}{\epsilon k}\right)^{4k}(9/\epsilon)^2\leq \left(\frac{9\sqrt{2}en}{\epsilon k}\right)^{4k+2}\]}}
provided with $n\geq k$ and $\epsilon\leq 1$.
\end{proof}
We now have the necessary ingredients to prove Theorem  \ref{thm: RIP}.
\begin{proof}[Proof of Theorem \ref{thm: RIP}]
Without loss of generality, we assume that  $\|X\|_F=1$. We first consider
$\mathbb{E}\|\mathcal{A}(X)\|_1$. Noting that $\text{rank}(X)\leq 2$ and $\|X\|_F=1$,
we can write $X$ in the form of
\[
X=\lambda_1\mathbf{u}_1\mathbf{u}_1^*+\lambda_2\mathbf{u}_2\mathbf{u}_2^*,
\]
where $\lambda_1,\lambda_2\in \R$ satisfying $\lambda_1^2+\lambda_2^2=1$ and
$\vu_1,\vu_2\in \C^n$ satisfying $\|\mathbf{u}_1\|_2=\|\mathbf{u}_2\|_2=1, \langle
\mathbf{u}_1,\mathbf{u}_2\rangle=0$. Therefore, we obtain that
\[
\va_k^*X\va_k=\lambda_1 \abs{\vu_1^*\va_k}^2+\lambda_2 \abs{\vu_2^*\va_k}^2,
\]
 where $\vu_1^*\va_k$ and $\vu_2^*\va_k$ are independently drawn from $\mathcal{N}(0,
\frac{1}{2})+\mathcal{N}(0,\frac{1}{2})i$.  Then
\begin{equation}\label{eq:AXxi}
\frac{1}{m}\|\mathcal{A}(X)\|_1=\frac{1}{m}\sum_{j=1}^m\left|\lambda_1 \abs{\vu_1^*\va_j}^2+\lambda_2
\abs{\vu_2^*\va_j}^2\right|=\frac{1}{m}\sum_{j=1}^m\xi_j,
\end{equation}
where the $\xi_j$ are independent copies of the following random variable
\[
\xi=\left|\lambda_1z_1^2+\lambda_1 z_2^2+\lambda_2 z_3^2+\lambda_2 z_4^2\right|
\]
where $z_1, z_2, z_3, z_4 \thicksim \mathcal{N}(0,\frac{1}{2})$ are independent.
Without loss of generality, we assume that $|\lambda_1|\geq |\lambda_2|$ and hence
$\abs{\lambda_1}\in [\frac{\sqrt{2}}{2},1]$. Note that  $\xi$ can also be rewritten
as
\begin{equation}\label{eq:xi}
\xi=|\lambda_1|\left|z_1^2+ z_2^2+t z_3^2+t z_4^2\right|
\end{equation}
with $t:=\lambda_2/\lambda_1$ satisfying $\abs{t}\leq 1$. Noting that
$\mathbb{E}\|\mathcal{A}(X)\|_1=\E(\xi)$, we next consider $\E (\xi)$. According to
(\ref{eq:xi}),  we have
\begin{equation}\label{eq:xiup}
\mathbb{E}(\xi)\leq \abs{\lambda_1} \mathbb{E} (z_1^2+z_2^2+z_3^2+z_4^2)\leq 2,
 \end{equation}
as $\mathbb{E}(z_j^2)=\frac{1}{2}$ for $j=1,\ldots,4$. On the other hand, when $t\geq
0$, we obtain that
\begin{equation}\label{eq:xilower1}
\mathbb{E}(\xi) \geq \abs{\lambda_1} \mathbb{E}(z_1^2+z_2^2)\geq \frac{\sqrt{2}}{2}.
\end{equation}
For $t\in [-1,0]$,  Lemma \ref{lem: expectation} shows that
\begin{equation}\label{eq:xilower}
\mathbb{E}(\xi)=\abs{\lambda_1}\left(\frac{1+t^2}{1-t}\right)\geq 0.57.
\end{equation}
 Combining (\ref{eq:xiup}), (\ref{eq:xilower1}) and (\ref{eq:xilower}), we obtain
 that
\[
0.57\,\,\leq\,\, \mathbb{E}(\xi)\,\,\leq\,\, 2.
\]

We next consider the bounds of $\|\A(X)\|_1$.
 Note that $\xi$ is a sub-exponential  variable with $\|\xi\|_{\psi_1}\leq \sum_{i=1}^4\|z_i\|_{\psi_1}\leq 4$.
 Here $\|x\|_{\psi_1}:=\sup_{p\geq 1} p^{-1}(\mathbb{E}|x|^p)^{1/p}$.
 To state conveniently, we set
\[
\mathcal{X}:=\{X\in \mathbb{H}^{n\times n}\ : \ \|X\|_F=1,\ \text{rank}(X)\leq 2,\ \|X\|_{0,2}\leq k\}.
\]

  We use $\mathcal{N}_{\epsilon}$ to denote an $\epsilon$-net of the set
${\mathcal X}$ respect to Frobenius norm $\|\cdot\|_F$, i.e. for any $X\in
\mathcal{X}$, there  exists $X_0\in \mathcal{N}_\epsilon$ such that $\|X-X_0\|_F\leq
\epsilon$.
 Based on Lemma \ref{Bernstein inequality} and equality (\ref{eq:AXxi}), we obtain that
\begin{equation}\label{eq:URIP}
0.57-\epsilon_0\,\,\leq\,\, \frac{1}{m}\|\mathcal{A}(X_0)\|_1\,\,\leq\,\, 2+\epsilon_0, \text{ for all } X_0\in \mathcal{N}_\epsilon
\end{equation}
holds with probability at least $1-2\cdot \# \mathcal{N}_\epsilon\cdot
\exp(-\frac{c_0}{16}m\epsilon_0^2)$.

  Assume  that $X\in \mathcal{X}$ with $X_0\in
\mathcal{N}_\epsilon$ such that $\|X-X_0\|_F\leq \epsilon$. Note that $\A$ is
continuous about $X\in \mathcal{X}$ and ${\mathcal X}$ is a compact set. We can set
\[
U_\A:=\max_{X\in {\mathcal X}} \frac{1}{m}\|\A(X)\|_1
\]
and
\[
U_0:=\max \{U_\A : \A \text{ satisfies } (\ref{eq:URIP})\}.
\]
 Note that ${\rm rank}(X-X_0)\leq 4$ and $\|X\|_{2,0}\leq k, \|X_0\|_{2,0}\leq k$.
We can decompose $X-X_0$ as $X-X_0=X_1+X_2$ where $X_1,X_2\in {\mathcal X}$ and
$\innerp{X_1,X_2}=0$ which leads to
\[
\begin{split}
\frac{1}{m}\|\mathcal{A}(X-X_0)\|_1&=\frac{1}{m}\|\mathcal{A}(X_1+X_2)\|_1\leq \frac{1}{m}\|\mathcal{A}(X_1)\|_1+\frac{1}{m}\|\mathcal{A}(X_2)\|_1\\
&\leq U_0\|X_1\|_F+U_0\|X_2\|_F\leq \sqrt{2}U_0\|X_1+X_2\|_F\leq \sqrt{2}U_0\epsilon.
\end{split}
\]
We obtain that
\begin{equation}\label{eq:upu0}
 \frac{1}{m}\|\mathcal{A}(X)\|_1\leq
\frac{1}{m}\|\mathcal{A}(X_0)\|_1+\frac{1}{m}\|\mathcal{A}(X-X_0)\|_1\leq
2+\epsilon_0+\sqrt{2}U_0\epsilon.
\end{equation}
According to the definition of $U_0$, (\ref{eq:upu0}) implies  $ U_0\leq
2+\epsilon_0+\sqrt{2}U_0\epsilon $ and hence
 which implies that
\[
U_0 \leq \frac{2+\epsilon_0}{1-\sqrt{2}\epsilon}.
\]
We also have
\[
\frac{1}{m}\|\mathcal{A}(X)\|_1\geq \frac{1}{m}\|\mathcal{A}(X_0)\|_1-\frac{1}{m}\|\mathcal{A}(X-X_0)\|_1\geq 0.57-\epsilon_0-\sqrt{2}U_0\epsilon
\geq 0.57-\epsilon_0-\sqrt{2}\frac{2+\epsilon_0}{1-\sqrt{2}\epsilon}\epsilon.
\]
 Hence, we obtain that the following holds with probability at
 least $1-2\cdot \# \mathcal{N}_\epsilon\cdot \exp(-\frac{c_0}{16}m\epsilon_0^2)$
\[
\left(0.57-\epsilon_0-\sqrt{2}\frac{2+\epsilon_0}{1-\sqrt{2}\epsilon}\epsilon\right)\|X\|_F\leq
 \frac{1}{m}\|\mathcal{A}(X)\|_1\leq \left(\frac{2+\epsilon_0}{1-\sqrt{2}\epsilon}\right)\|X\|_F, \text{ for all } X\in \mathcal{X}.
\]
 Taking $\epsilon=\epsilon_0=0.1$,  according to Lemma \ref{lem: covering number}, we obtain
 { {$\# \mathcal{N}_\epsilon\leq \left(\frac{90\sqrt{2}en}{k}\right)^{4k+2}$}.} Thus when $m\geq O(k\log(en/k))$, we obtain that
\[
0.12\|X\|_F\leq \frac{1}{m}\|\mathcal{A}(X)\|_1\leq 2.45\|X\|_F,\quad \text{ for all } X\in \mathcal{X}
\]
 holds with probability at least $1-2\exp(-cm)$.
\end{proof}

\section{Proof of Theorem \ref{thm: noise_constrained}}

We first introduce the  convex    $k$-sparse decomposition of signals which was
proved independently  in \cite{CZ14} and \cite{XX}.
\begin{lemma}\cite{CZ14,XX}\label{eqn: TonyCai}
Suppose that $\mathbf{v}\in \mathbb{R}^p$ satisfying $\|\mathbf{v}\|_\infty\leq
\theta$, $\|\mathbf{v}\|_1\leq s\theta$ where $\theta>0$ and $s\in \Z_+$. Then we
have
\[
\mathbf{v}=\sum_{i=1}^N\lambda_i\mathbf{u}_i,\qquad 0\leq \lambda_i\leq 1,\qquad \sum_{i=1}^N
\lambda_i=1,
\]
where $\mathbf{u}_i$ is $s$-sparse with $(\text{supp}(\mathbf{u}_i))\subset
\text{supp}(\mathbf{v})$, and
\[
\|\mathbf{u}_i\|_1\leq \|\mathbf{v}\|_1,\qquad \|\mathbf{u}_i\|_\infty\leq \theta.
\]

\end{lemma}
We also need the following lemma:
\begin{lemma}\label{lem: MatrixToVector}
If $\vx, \vy\in\mathbb{C}^{d}$, and $\innerp{\vx,\vy} \geq 0$, then
\[
\|\vx\vx^{*}-\vy\vy^{*}\|_{F}^{2}\geq\frac{1}{2}\|\vx\|_{2}^{2}\|\vx-\vy\|_{2}^{2}.
\]
Similarly, we have
\[
\|\vx\vx^{*}-\vy\vy^{*}\|_{F}^{2}\geq\frac{1}{2}\|\vy\|_{2}^{2}\|\vx-\vy\|_{2}^{2}.
\]
\end{lemma}
\begin{proof}
To state conveniently, we set $a:=\|\vx\|_{2}, \ b:=\|\vy\|_{2}$ and $t:=\frac{\langle
\vx,\vy\rangle}{\|\vx\|_{2}\|\vy\|_{2}}$. A simple calculation shows that
\[
\|\vx\vx^{*}-\vy\vy^{*}\|_{F}^{2}-\frac{1}{2}\|\vx\|_{2}^{2}\|\vx-\vy\|_{2}^{2}\,\,=\,\,h(a,b,t)
\]
where
\[
h(a,b,t):=a^{4}+b^{4}-2(ab)^{2}t^{2}-\frac{1}{2}a^{2}(a^{2}+b^{2}-2abt).
\]
Hence, to this end, it is enough to show that  $h(a,b,t)\geq 0$ provided $a,b\geq 0$
and $0\leq t\leq 1$. For any fixed $a$ and $b$, $h(a,b,t)$ achieves the minimum for
either $t=0$ or $t=1$. For $t=0$, we have
\begin{equation}
h(a,b,0)=a^{4}+b^{4}-\frac{1}{2}a^{4}-\frac{1}{2}a^{2}b^{2}=\frac{1}{2}(a^{2}-\frac{1}{2}b^{2})^{2}+\frac{7}{8}b^{4}\geq0.\label{eq: tech_inequa1}
\end{equation}
When $t=1$, we have
\begin{equation}
\label{eq: tech_inequa2}
\begin{split}
h(a,b,1)&=a^{4}+b^{4}-\frac{1}{2}a^{2}(a^{2}+b^{2})-2(ab)^{2}+a^{3}b\\
&=(a-b)^2(\frac{1}{2}a^2+b^2+2ab)\geq 0
\end{split}
\end{equation}
Combining (\ref{eq: tech_inequa1}) and (\ref{eq: tech_inequa2}), we arrive at the
conclusion.
\end{proof}

Now we have enough ingredients to prove Theorem  \ref{thm: noise_constrained}.

\begin{proof}[Proof of Theorem \ref{thm: noise_constrained}]
 We assume that $\vx^\#$ is a solution to (\ref{eq:complexL1}).
Noting $\exp(i\theta)\vx^\#$ is also a solution to (\ref{eq:complexL1}) for any
$\theta\in \R$, without loss of generality, we  can assume that
\[
\innerp{\vx^\#,\vx_{0}}\in \R\qquad\text{and}\qquad\innerp{\vx^\#,\vx_{0}}\geq 0.
\]
We  consider the programming
\begin{equation}\label{eq: model_rank1}
\min_{X\in {\mathbb H}^{n\times n}} \|X\|_1\quad {s.t.}\quad \|\A(X)-\vy\|_2\leq \epsilon, \ {{{\rm rank}(X)=1}.}
\end{equation}
Then a simple observation  is that $X^{\#}$ is the solution to (\ref{eq: model_rank1}) if and only if $X^\#=\vx^\# (\vx^\#)^*$.

Set $X_0:=\vx_0\vx_0^*$ and $H:=X^\#-X_0=\vx^\#(\vx^\#)^{*}-\vx_{0}\vx_{0}^{*}$. To this end, it is enough to consider the upper bound of $\|H\|_F$.  Denote $T_{0}=\text{supp}(\vx_{0})$.
 Set $T_{1}$ as the index set which contains the indices of the $ak$ largest
 elements of  $\vx^\#_{T_{0}^{c}}$ in magnitude, and $T_2$ contains the indices of the next $ak$ largest elements, and so on. For simplicity, we set $T_{01}:=T_{0}\cup T_{1}$
and  $\bar{H}:=H_{T_{01},T_{01}}$, where $H_{S,T}$ denotes the
sub-matrix of $H$ with the row set $S$ and the column set $T$. Therefore, it is enough to consider $\|H\|_F$. We claim that
\begin{equation}\label{eq:mainH}
\|H\|_F\leq \|\bar{H}\|_F+\|H-\bar{H}\|_F\leq \left(\frac{1}{a}+\frac{4}{\sqrt{a}}+1\right)\|\bar{H}\|_F\leq \frac{\frac{1}{a}+\frac{4}{\sqrt{a}}+1}{c-\frac{4C}{\sqrt{a}}-\frac{C}{a}}\frac{2\epsilon}{\sqrt{m}},
\end{equation}
which implies the conclusion (\ref{eqn: conclusion1}). {{According to Lemma \ref{lem:
MatrixToVector}, we obtain that
\[
\min_{c\in \C, \abs{c}=1}\|c\cdot\vx^\# -\vx_0\|_2\leq \|\vx^\# -\vx_0\|_2\leq \sqrt{2}\|H\|_F/\|\vx_0\|_2\leq \frac{\frac{1}{a}+\frac{4}{\sqrt{a}}+1}{c-\frac{4C}{\sqrt{a}}-\frac{C}{a}}\frac{2\sqrt{2}\epsilon}{\sqrt{m}\|\vx_0\|_2}.
\]
Furthermore, we also have
\[
\min_{c\in \C, \abs{c}=1}\|c\cdot\vx^\# -\vx_0\|_2\leq \|\vx^\# -\vx_0\|_2\leq \|\vx^\#\|_2+\|\vx_0\|_2\leq \|\vx^\#\|_1+\|\vx_0\|_2\leq \|\vx_0\|_1+\|\vx_0\|_2.
\]
Here, we use $\|\vx^\#\|_1\leq \|\vx_0\|_1$. Combining the above two inequalities, we
obtain that
\[
\min_{c\in \C, \abs{c}=1}\|c\cdot\vx^\# -\vx_0\|_2\leq \min\left\{
 2\sqrt{2}C_1\frac{\epsilon}{\sqrt{m}\|\vx_0\|_2},\|\vx_0\|_2+\|\vx_0\|_1\right\}.
\]
For the case where $\|\vx_0\|_2+\|\vx_0\|_1\leq
2\sqrt{2}C_1\frac{\epsilon}{\sqrt{m}\|\vx_0\|_2}$, we obtain that
\begin{equation*}
\begin{aligned}
2\sqrt{2}C_1\frac{\epsilon}{\sqrt{m}}&\geq \|\vx_0\|_2^2 +\|\vx_0\|_2\|\vx_0\|_1\\
& \geq \|\vx_0\|_1^2/n+\|\vx_0\|_1^2/\sqrt{n}= \|\vx_0\|_1^2 (1/n+1/\sqrt{n}),
\end{aligned}
\end{equation*}
which implies $\|\vx_0\|_1 \leq \sqrt{2\sqrt{2} C_1}\cdot \sqrt{\epsilon}\cdot
{(n/m)}^{1/4}$. Hence, when $\|\vx_0\|_2+\|\vx_0\|_1\leq
2\sqrt{2}C_1\frac{\epsilon}{\sqrt{m}\|\vx_0\|_2}$, we have
\[
\|\vx_0\|_2+\|\vx_0\|_1 \leq 2\|\vx_0\|_1 \leq
2 \sqrt{2\sqrt{2} C_1}\cdot \sqrt{\epsilon}\cdot (n/m)^{1/4}.
\]
We arrive at the conclusion (\ref{eqn: conclusion2}).
 }}

We next turn to prove (\ref{eq:mainH}). The first inequality in (\ref{eq:mainH})
follows from
\begin{equation}\label{eq:main1}
\|H-\bar{H}\|_F\,\,\leq \,\, \left(\frac{1}{a}+\frac{4}{\sqrt{a}}\right) \|\bar{H}\|_F
\end{equation}
and the second inequality follows from
\begin{equation}\label{eq:main2}
\|\bar{H}\|_F\leq  \frac{1}{c-\frac{4C}{\sqrt{a}}-\frac{C}{a}}\frac{2\epsilon}{\sqrt{m}}.
\end{equation}
To this end, it is enough to prove (\ref{eq:main1}) and (\ref{eq:main2}).


\textbf{Step 1}: We first present the proof of (\ref{eq:main1}).
A simple observation is that
\begin{equation}\label{eq:right1}
\begin{aligned}
\|H-\bar{H}\|_F&\leq \sum_{i\geq 2,j\geq 2}\|H_{T_i,T_j}\|_F+\sum_{i=0,1}\sum_{j\geq 2}\|H_{T_i,T_j}\|_F+\sum_{j=0,1}\sum_{i\geq 2}\|H_{T_i,T_j}\|_F\\
&= \sum_{i\geq 2,j\geq 2}\|H_{T_i,T_j}\|_F+2\sum_{i=0,1}\sum_{j\geq 2}\|H_{T_i,T_j}\|_F.
\end{aligned}
\end{equation}
We first consider the first term on the right side of (\ref{eq:right1}).
Note that
\begin{equation}\label{eq:term1}
\begin{aligned}
\sum_{i\geq2,j\geq2}\|H_{T_{i},T_{j}}\|_{F} & = \sum_{i\geq2,j\geq2}\|\vx^\#_{T_{i}}\|_{2}\cdot \|\vx^\#_{T_{j}}\|_{2}
=\left(\sum_{i\geq2}\|\vx^\#_{T_{i}}\|_{2}\right)^{2}\leq\frac{1}{ak}\|\vx^\#_{T_{0}^{c}}\|_{1}^{2}\\
&=\frac{1}{ak}\|H_{T_{0}^{c},T_{0}^{c}}\|_{1}
  \leq  \frac{1}{ak}\|H_{T_{0},T_{0}}\|_{1}\leq\frac{1}{a}\|H_{T_{0},T_{0}}\|_{F}\leq\frac{1}{a}\|\bar{H}\|_{F}.
 \end{aligned}
\end{equation}
Here, the first inequality follows from $\|\vx^\#_{T_i}\|_2\leq \|\vx^\#_{T_{i-1}}\|_1/\sqrt{ak}$, for $i\geq 2$.
The second inequality is based on $\|H-H_{T_{0},T_{0}}\|_{1}\leq \|H_{T_{0},T_{0}}\|_{1}$.
Indeed, according to $\|X^\#\|_1\leq \|X_0\|_1$, we have
\[
\|H-H_{T_{0},T_{0}}\|_{1}=\|X^\#-X^\#_{T_0,T_0}\|_1\leq \|X_0\|_1-\|X_{T_0,T_0}^\#\|_1\leq \|X_0-X_{T_0,T_0}^\#\|_1=\|H_{T_0,T_0}\|_1.
\]
We next turn to $\sum_{i=0,1}\sum_{j\geq 2}\|H_{T_i,T_j}\|_F$. When $i\in \{0,1\}$, noting that $\|\vx^\#_{T_j}\|_2\leq \|\vx^\#_{T_{j-1}}\|_1/\sqrt{ak}$, we have
\begin{equation}\label{eq:term3}
\begin{aligned}
\sum_{j\geq2}\|H_{T_{i},T_{j}}\|_{F}=\|\vx^\#_{T_{i}}\|_{2} \cdot\sum_{j\geq2}\|\vx^\#_{T_{j}}\|_{2}
\leq\frac{1}{\sqrt{ak}}\|\vx^\#_{T_{0}^{c}}\|_{1}\|\vx^\#_{T_{i}}\|_{2}
\leq \frac{1}{\sqrt{a}}\|\vx^\#_{T_{i}}\|_{2}\|\vx^\#_{T_{01}}-\vx_0\|_{2}.
\end{aligned}
\end{equation}
The last inequality is based on $\|\vx^\#\|_1\leq \|\vx_0\|_1$, which leads to
\[
\|\vx^\#_{T_{0}^c}\|_{1}\leq \|\vx_0\|_1-\|\vx^\#_{T_{0}}\|_1\leq \|\vx^\#_{T_{0}}-\vx_0\|_1\leq \sqrt{k}\|\vx^\#_{T_{0}}-\vx_0\|_2\leq \sqrt{k}\|\vx^\#_{T_{01}}-\vx_0\|_2.
\]
 Substituting (\ref{eq:term1}) and (\ref{eq:term3}) into (\ref{eq:right1}), we obtain that
\begin{equation}
\label{eqn: final1}
\begin{split}
\|H-\bar{H}\|_F &\leq \sum_{i\geq 2,j\geq 2}\|H_{T_i,T_j}\|_F+\sum_{i=0,1}\sum_{j\geq 2}\|H_{T_i,T_j}\|_F+\sum_{j=0,1}\sum_{i\geq 2}\|H_{T_i,T_j}\|_F\\
&\leq\frac{1}{a}\|\bar{H}\|_{F}+\frac{2\sqrt{2}}{\sqrt{a}}\|\vx^\#_{T_{01}}\|_{2}\|\vx^\#_{T_{01}}-\vx_{0}\|_{2}\leq \left(\frac{1}{a}+\frac{4}{\sqrt{a}}\right)\|\bar{H}\|_F.
\end{split}
\end{equation}
Here, the first inequality is based on $\|\vx^\#_{T_{0}}\|_{2}+\|\vx^\#_{T_{1}}\|_{2}\leq \sqrt{2}\|\vx^\#_{T_{01}}\|_{2}$, and the second inequality follows from Lemma \ref{lem: MatrixToVector}.

\textbf{Step 2}: We next prove (\ref{eq:main2}). Since  \[
 \|\A(H)\|_2\leq \|\mathcal{A}(X^{\#})-\vy\|_2+\|\mathcal{A}(X_0)-\vy\|_2\leq 2\epsilon,
 \]
 we have
 \begin{equation}
 \label{eqn: error estimation1}
 \frac{{2}\epsilon}{\sqrt{m}}\geq \frac{1}{\sqrt{m}} \|\A(H)\|_2\geq \frac{1}{m}\|\A(H)\|_1\geq \frac{1}{m}\|\A(\bar{H})\|_1-\frac{1}{m}\|\A(H-\bar{H})\|_1.
 \end{equation}
In order to get the lower bound of $\frac{1}{m}\|\A(\bar{H})\|_1-\frac{1}{m}\|\A(H-\bar{H})\|_1$,  we need to estimate the lower bound of $\frac{1}{m}\|\A(\bar{H})\|_1$ and the upper bound of $\frac{1}{m}\|\A(H-\bar{H})\|_1$. As $\text{rank}(\bar{H})\leq 2$ and $\|\bar{H}\|_{0,2}\leq 2k$,
 based on the RIP of measurement $\A$, we obtain that
\begin{equation}\label{eqn: lower_final1}
\frac{1}{m}\|\mathcal{A}(\bar{H})\|_1\geq c\|\bar{H}\|_F.
\end{equation}
We next consider an upper bound of $\frac{1}{m}\|\A(H-\bar{H})\|_1$. Since
$H-\bar{H}$ can be written as
\[
H-\bar{H}=(H_{T_{0},T_{01}^c}+H_{T_{01}^c,T_{0}})+(H_{T_{1},T_{01}^c}+H_{T_{01}^c,T_{1}})+H_{T_{01}^c,T_{01}^c},
\]
we have
\begin{equation}\label{mid: final}
\frac{1}{m}\|\A(H-\bar{H})\|_1\leq \frac{1}{m}\|\A(H_{T_{0},T_{01}^c}+H_{T_{01}^c,T_{0}})\|_1+\frac{1}{m}\|\A(H_{T_{1},T_{01}^c}+H_{T_{01}^c,T_{1}})\|_1+\frac{1}{m}\|\A(H_{T_{01}^c,T_{01}^c})\|_1.
\end{equation}
We next calculate the upper bound of
$\frac{1}{m}\|\A(H_{T_{i},T_{01}^c}+H_{T_{01}^c,T_{i}})\|_1, i\in \{0,1\}$, and
$\frac{1}{m}\|\A(H_{T_{01}^c,T_{01}^c})\|_1$, respectively. According to the RIP
condition, for $i\in \{0,1\}$, we have
\begin{equation}\label{eqn: mid1}
\begin{split}
\frac{1}{m}\|\A(H_{T_{i},T_{01}^c}+H_{T_{01}^c,T_{i}})\|_1&\leq \sum_{j\geq 2}\frac{1}{m}\|\A(H_{T_{i},T_j}+H_{T_{j},T_{i}})\|_1\leq \sum_{j\geq 2}C\|H_{T_{i},T_j}+H_{T_{j},T_{i}}\|_F\\
& \leq C \sum_{j\geq 2}(\|\vx_{T_{i}}^\#(\vx_{T_j}^\#)^*\|_F+\|\vx_{T_{j}}^\#(\vx_{T_{i}}^\#)^*\|_F)=2C\sum_{j\geq 2}\|\vx_{T_{i}}^\#\|_2\|\vx_{T_j}^\#\|_2\\
&\leq \frac{2C}{\sqrt{a}} \|\vx^\#_{T_{i}}\|_{2}\|\vx^\#_{T_{01}}-\vx_0\|_{2}.
\end{split}
\end{equation}
Here, the first inequality follows from
\[
H_{T_{i},T_{01}^c}+H_{T_{01}^c,T_{i}}=\sum_{j\geq 2}(H_{T_{i},T_j}+H_{T_{j},T_{i}})=\sum_{j\geq 2}(\vx_{T_{i}}^\#(\vx_{T_j}^\#)^*+\vx_{T_{j}}^\#(\vx_{T_{i}}^\#)^*)
\]
and the last inequality is obtained by (\ref{eq:term3}). Then it remains to estimate
the upper bound of $\frac{1}{m}\|\A(H_{T_{01}^c,T_{01}^c})\|_1$. Note that
\[
H_{T_{01}^c,T_{01}^c}=\vx_{T_{01}^c}^\#(\vx_{T_{01}^c}^\#)^*
\]
 with $\|\vx_{T_{01}^c}^\#\|_\infty\leq \|\vx_{T_{1}}^\#\|_1/(ak)$.
 Set $\theta:=\max\{\| \vx_{T_{1}}^\#\|_1/(ak),\|\vx_{T_{01}^c}^\#\|_1/(ak)\}$.
 We assume that  $\Phi:={\rm Diag}(Ph (\vx_{T_{01}^c}^\#))$ is the diagonal matrix
 with diagonal  elements being the phase of $\vx_{T_{01}^c}^\#$ and then
 $\Phi^{-1}\vx_{T_{01}^c}^\#$ is a real vector.
 According to Lemma \ref{eqn: TonyCai}, we have
\[
\Phi^{-1}\vx_{T_{01}^c}^\#=\sum_{i=1}^N\lambda_i\mathbf{u}_i,\qquad 0\leq \lambda_i\leq 1,\qquad \sum_{i=1}^N
\lambda_i=1,
\]
where $\mathbf{u}_i$ is $ak$-sparse, and
\[
\|\mathbf{u}_i\|_1\leq \|\vx_{T_{01}^c}^\#\|_1,\qquad \|\mathbf{u}_i\|_\infty\leq \theta,
\]
which leads to
\[
\|\mathbf{u}_i\|_2\leq \sqrt{\|\mathbf{u}_i\|_1\|\mathbf{u}_i\|_\infty}\leq \sqrt{\theta\|\vx_{T_{01}^c}^\#\|_1}.
\]
If $\theta=\|\vx_{T_{1}}^\#\|_1/(ak)$,  we have
\[
\begin{split}
\|\mathbf{u}_i\|_2&\leq \sqrt{\frac{\|\vx_{T_{1}}^\#\|_1\|\vx_{T_{01}^c}^\#\|_1}{ak}}= \sqrt{\frac{\|H_{T_1,T_{01}^c}\|_1}{ak}}\\
&\leq \sqrt{\frac{\|H-H_{T_0,T_0}\|_1}{ak}}\leq \sqrt{\frac{\|H_{T_0,T_0}\|_1}{ak}}\leq \sqrt{\frac{\|H_{T_0,T_0}\|_F}{a}}\leq \sqrt{\frac{\|\bar{H}\|_F}{a}}.\\
\end{split}
\]
If $\theta=\|\vx_{T_{01}^c}^\#\|_1/(ak)$,  we have
\[
\label{eqn: separate upper2}
\begin{split}
\|\mathbf{u}_i\|_2&\leq \sqrt{\frac{\|\vx_{T_{01}^c}^\#\|_1\|\vx_{T_{01}^c}^\#\|_1}{ak}}= \sqrt{\frac{\|H_{T_{01}^c,T_{01}^c}\|_1}{ak}}\\
&\leq \sqrt{\frac{\|H-H_{T_0,T_0}\|_1}{ak}}\leq \sqrt{\frac{\|H_{T_0,T_0}\|_1}{ak}}\leq \sqrt{\frac{\|H_{T_0,T_0}\|_F}{a}}\leq \sqrt{\frac{\|\bar{H}\|_F}{a}}.\\
\end{split}
 \]
 Thus we can obtain that
 \begin{equation}
 \label{eqn: upper_ui}\|\mathbf{u}_i\|_2\leq \sqrt{\frac{\|\bar{H}\|_F}{a}}, \text{ for } i=1,\ldots,N.
 \end{equation}
  Since
 \[
 \begin{aligned}
 H_{T_{01}^c,T_{01}^c}&=\vx_{T_{01}^c}^\#(\vx_{T_{01}^c}^\#)^*
 =\left(\sum_{i=1}^N\lambda_i\Phi\mathbf{u}_i\right)\left(\sum_{i=1}^N\lambda_i\Phi\mathbf{u}_i\right)^*\\
 &=\sum_{i<j}\lambda_i\lambda_j\Phi(\mathbf{u}_i\mathbf{u}_j^*+\mathbf{u}_j\mathbf{u}_i^*)\Phi^{-1}
 +\sum_{i}\lambda_i^2\Phi\mathbf{u}_i\mathbf{u}_i^*\Phi^{-1},
 \end{aligned}
 \]
 based on the RIP condition, we can obtain that
 \begin{equation}
 \label{eqn: mid2}
 \begin{split}
 \frac{1}{m}\|\A(H_{T_{01}^c,T_{01}^c})\|_1&\leq \sum_{i<j}C\lambda_i\lambda_j\|(\mathbf{u}_i\mathbf{u}_j^*+\mathbf{u}_j\mathbf{u}_i^*)\|_F+\sum_{i}C\lambda_i^2\|\mathbf{u}_i\mathbf{u}_i^*\|_F\\
 &\leq \sum_{i<j}2C\lambda_i\lambda_j\|\mathbf{u}_i\|_2\|\mathbf{u}_j\|_2+\sum_{i}C\lambda_i^2\|\mathbf{u}_i\|_2^2\\
 &\leq C\frac{\|\bar{H}\|_F}{a}\left(\sum_i\lambda_i\right)^2=C\frac{\|\bar{H}\|_F}{a}
 \end{split}
 \end{equation}
Here, the third line follows from (\ref{eqn: upper_ui}).  Now combing  (\ref{eqn:
mid1}) and (\ref{eqn: mid2}),  we obtain that
\begin{equation}
\label{eqn: lower_final2}
\begin{split}
\frac{1}{m}\|\A(H-\bar{H})\|_1&\leq \frac{1}{m}\|\A(H_{T_{0},T_{01}^c}+H_{T_{01}^c,T_{0}})\|_1+\frac{1}{m}\|\A(H_{T_{1},T_{01}^c}+H_{T_{01}^c,T_{1}})\|_1+\frac{1}{m}\|\A(H_{T_{01}^c,T_{01}^c})\|_1\\
&\leq \frac{2C}{\sqrt{a}} \|\vx^\#_{T_{0}}\|_{2}\|\vx^\#_{T_{01}}-\vx_0\|_{2}+\frac{2C}{\sqrt{a}} \|\vx^\#_{T_{1}}\|_{2}\|\vx^\#_{T_{01}}-\vx_0\|_{2}+C\frac{\|\bar{H}\|_F}{a}\\
&\leq \frac{2\sqrt{2}C}{\sqrt{a}} \|\vx^\#_{T_{01}}\|_{2}\|\vx^\#_{T_{01}}-\vx_0\|_{2}+C\frac{\|\bar{H}\|_F}{a}\\
&\leq C\left(\frac{4}{\sqrt{a}}+\frac{1}{a}\right)\|\bar{H}\|_F.
\end{split}
\end{equation}
 The last inequality uses Lemma \ref{lem: MatrixToVector}.
Based on  (\ref{eqn: lower_final1}), (\ref{eqn: lower_final2}) and (\ref{eqn: error
estimation1}), we obtain that
 \[
 \begin{split}
 \frac{{2}\epsilon}{\sqrt{m}}&\geq \frac{1}{m}\|\A(\bar{H})\|_1-\frac{1}{m}\|\A(H-\bar{H})\|_1\\
 &\geq c\|\bar{H}\|_F-C\left(\frac{4}{\sqrt{a}}+\frac{1}{a}\right)\|\bar{H}\|_F=\left(c-\frac{4C}{\sqrt{a}}-\frac{C}{a}\right)\|\bar{H}\|_F.
\end{split}
 \]
 According to the condition (\ref{eqn: condition_a}), it implies that
 \[
 \|\bar{H}\|_F\leq \frac{1}{c-\frac{4C}{\sqrt{a}}-\frac{C}{a}}\frac{{2}\epsilon}{\sqrt{m}}.
 \]
 Thus, we arrive at the conclusion (\ref{eq:main2}).
\end{proof}

\end{document}